\documentclass{amsproc}
\usepackage{breqn,float,amssymb,pdflscape,rotating}
\makeatletter
\@namedef{subjclassname@2020}{%
  \textup{2020} Mathematics Subject Classification}
\makeatother
\newtheorem{theorem}{Theorem}[section]

\theoremstyle{definition}

\newtheorem{example}[theorem]{Example}

\theoremstyle{remark}

\numberwithin{equation}{section}

\begin{document}

\title[Sum of the Hurwitz-Lerch Zeta Function]{Sum of the Hurwitz-Lerch Zeta Function over Prime Numbers: Derivation and Evaluation}


\author{Robert Reynolds}
\address[Robert Reynolds]{Department of Mathematics and Statistics, York University, Toronto, ON, Canada, M3J1P3}
\email[Corresponding author]{milver@my.yorku.ca}
%
\author{ Allan Stauffer}
\address[Allan Stauffer]{Department of Mathematics and Statistics, York University, Toronto, ON, Canada, M3J1P3}
\email{stauffer@yorku.ca}
\thanks{This research is supported by NSERC Canada under Grant 504070}



\date{}

\dedicatory{}

\begin{abstract}
For the function $f(m,p,q,n)$, where $k,s,a$ general complex numbers and $q$ any positive integer, we establish the sum of values of the Hurwitz-Lerch zeta function $\Phi(f(m,p,q,n),k,a)$ taken at prime numbers $n$. Special cases of this sum are evaluated in terms of products of trigonometric functions and Catalan's constant $K$.
\end{abstract}

\maketitle
\section{Introduction}
The sum of the Hurwitz-Lerch Zeta Function has been published in the work by Garunk\u{s}tis \cite{garunktis} where the sum of the Lerch Zeta-Function over nontrivial zeros of the Dirichlet $L$-Function was evaluated. In this present work we derive a new expression for the Hurwitz-Lerch zeta function in terms of the sum of two Hurwitz-Lerch zeta functions given by
\begin{dmath}
\sum_{p=0}^{n-1}\left(\log ^k(a)-i^k 2^{k+1} e^{\frac{2 i (m n+\pi  p q)}{n}} \Phi \left(-e^{2 i \left(m+\frac{p \pi 
   q}{n}\right)},-k,1-\frac{1}{2} i \log (a)\right)\right)\\\\
=   n \left(\log ^k(a)-2^{k+1} (i n)^k e^{2 i m n} \Phi
   \left(-e^{2 i m n},-k,1-\frac{i \log (a)}{2 n}\right)\right)
\end{dmath}
where the variables $k,a,m$ are general complex numbers, $q$ is any positive integer and $n$ is any prime number. The derivations follow the method used by us in \cite{reyn4}. This method involves using a form of the generalized Cauchy's integral formula given by
\begin{equation}\label{intro:cauchy}
\frac{y^k}{\Gamma(k+1)}=\frac{1}{2\pi i}\int_{C}\frac{e^{wy}}{w^{k+1}}dw,
\end{equation}
where $y,w\in\mathbb{C}$ and $C$ is in general an open contour in the complex plane where the bilinear concomitant \cite{reyn4} has the same value at the end points of the contour. This method involves using a form of equation (\ref{intro:cauchy}) then multiplies both sides by a function, then takes the definite integral of both sides. This yields a definite integral in terms of a contour integral. Then we multiply both sides of equation (\ref{intro:cauchy})  by another function and take the infinite sum of both sides such that the contour integral of both equations are the same.
\section{The Lerch function}

We use equation (1.11.3) in \cite{erd} where $\Phi(z,s,v)$ is the Lerch function which is a generalization of the Hurwitz zeta $\zeta(s,v)$ and Polylogarithm functions $Li_{n}(z)$. The Lerch function has a series representation given by

\begin{equation}\label{knuth:lerch}
\Phi(z,s,v)=\sum_{n=0}^{\infty}(v+n)^{-s}z^{n}
\end{equation}
where $|z|<1, v \neq 0,-1,-2,-3,..,$ and is continued analytically by its integral representation given by

\begin{equation}\label{knuth:lerch1}
\Phi(z,s,v)=\frac{1}{\Gamma(s)}\int_{0}^{\infty}\frac{t^{s-1}e^{-(v-1)t}}{e^{t}-z}dt
\end{equation}
where $Re(v)>0$, and either $|z| \leq 1, z \neq 1, Re(s)>0$, or $z=1, Re(s)>1$.
\section{Finite Sum of the Contour Integral}
We use the method in \cite{reyn4}. The cut and contour are in the first quadrant of the complex $w$-plane with $0 < Re(w+m) <1$.  The cut approaches the origin from the interior of the first quadrant and goes to infinity vertically and the contour goes round the origin with zero radius and is on opposite sides of the cut. Using a generalization of Cauchy's integral formula (\ref{intro:cauchy}) we first replace $y \to \log (a)+2 i (y+1)$ then multiply both sides by $-2 i (-1)^y e^{2 i (y+1) \left(m+\frac{\pi  p q}{n}\right)}$ and take the infinite sums over $y\in[0,\infty)$ and $p\in[0,n-1]$ and simplify in terms of the Hurwitz-Lerch zeta function to get
\begin{multline}\label{fsci}
\sum_{p=0}^{n-1}i \left(\log ^k(a)-i^k 2^{k+1} e^{\frac{2 i (m n+\pi  p q)}{n}} \Phi \left(-e^{2 i \left(m+\frac{p \pi 
   q}{n}\right)},-k,1-\frac{1}{2} i \log (a)\right)\right)\\\\
   =-\frac{1}{2\pi i}\sum_{y=0}^{\infty}\sum_{p=0}^{n-1}\int_{C}2 i (-1)^y a^w w^{-k-1} e^{\frac{2 i (y+1) (n (m+w)+\pi  p q)}{n}}\\\\
   =-\frac{1}{2\pi i}\int_{C}\sum_{p=0}^{n-1}\sum_{y=0}^{\infty}2 i (-1)^y a^w w^{-k-1} e^{\frac{2 i (y+1) (n (m+w)+\pi  p q)}{n}}\\\\
   =\frac{1}{2\pi i}\int_{C}\sum_{p=0}^{n-1}\left(a^w w^{-k-1} \tan \left(m+\frac{\pi  p q}{n}+w\right)-i a^w
   w^{-k-1}\right)dw\\\\
   =\frac{1}{2\pi i}\int_{C}\left(n a^w w^{-k-1} \tan (n (m+w))dw-i n a^w w^{-k-1}\right)dw
\end{multline}
from equation (4.4.7.1) in \cite{prud1} where $Im(n (m+w))>0, Re(n (m+w))>0, q\in\mathbb{Z_{+}}$, $n$ is a prime number and $n \neq q$ in order for the sums to converge. We apply Tonelli's theorem for multiple sums, see page 189 in \cite{gelca} as the summand is of bounded measure over the space $\mathbb{C} \times [0,n-1] \times [0,\infty) $.
\subsection{The Additional Contour Integral}
Using a generalization of Cauchy's integral formula (\ref{intro:cauchy}) we first replace $y \to \log (a)$ and multiply both sides by $-in$ simplify to get
\begin{equation}\label{addc1}
-\frac{i n \log ^k(a)}{\Gamma(k+1)}=-\frac{1}{2\pi i}\int_{C}i n a^w w^{-k-1}dw
\end{equation}
\section{Infinite Sum of the Contour Integral}
\subsection{Derivation of the contour integral}
We use the method in \cite{reyn4}. Using a generalization of Cauchy's integral formula (\ref{intro:cauchy}) we first replace $y \to \log (a)+2 i n (y+1))$ then multiply both sides by $2 i n e^{2 i (y+1) \left(m n+\frac{\pi }{2}\right)}$ and take the infinite sums over $y\in[0,\infty)$ and simplify in terms of the Hurwitz-Lerch zeta function to get
\begin{multline}\label{isci}
\frac{2^{k+1} (i n)^{k+1} e^{2 i \left(m n+\frac{\pi }{2}\right)} \Phi \left(e^{2 i \left(m n+\frac{\pi
   }{2}\right)},-k,1-\frac{i \log (a)}{2 n}\right)}{\Gamma(k+1)}\\\\
   =\frac{1}{2\pi i}\sum_{y=0}^{\infty}\int_{C}2 i n a^w
   w^{-k-1} e^{2 i (y+1) (b+n (m+w))}dw\\\\
   =\frac{1}{2\pi i}\int_{C}\sum_{y=0}^{\infty}2 i n a^w
   w^{-k-1} e^{2 i (y+1) (b+n (m+w))}dw\\\\
=\frac{1}{2\pi i}\int_{C}  \left( n a^w w^{-k-1} \tan (n (m+w))-i n a^w w^{-k-1}\right)dw
\end{multline}
from equation (1,232,1) in \cite{grad} where $Im(n (m+w))>0$ in order for the sum to converge.
\subsection{The Additional Contour Integral}
Using a generalization of Cauchy's integral formula (\ref{intro:cauchy}) we first replace $y \to \log (a)$ and multiply both sides by $-in$ simplify to get
\begin{equation}\label{addc2}
-\frac{i n \log ^k(a)}{\Gamma(k+1)}=-\frac{1}{2\pi i}\int_{C}i n a^w w^{-k-1}dw
\end{equation}
\section{Sum of the Hurwitz-Lerch zeta Function}
\begin{theorem}
For all $k,a,m\in\mathbb{C}, q\in\mathbb{Z_{+}}$, $n$ any prime number and $n \neq q$ then  
\begin{multline}\label{fslf}
\sum_{p=0}^{n-1} \left(\log ^k(a)-i^k 2^{k+1} e^{\frac{2 i (m n+\pi  p q)}{n}} \Phi \left(-e^{2 i \left(m+\frac{p \pi 
   q}{n}\right)},-k,1-\frac{1}{2} i \log (a)\right)\right)\\\\
=   n \left(\log ^k(a)-2^{k+1} (i n)^k e^{2 i m n} \Phi
   \left(-e^{2 i m n},-k,1-\frac{i \log (a)}{2 n}\right)\right)
\end{multline}
\end{theorem}
\begin{proof}
Observe that the addition of the right-hand sides of equations (\ref{fsci}) and (\ref{addc1}), is equal to the addition of the right-hand sides of equations (\ref{isci}) and (\ref{addc2}) so we may equate the left-hand sides and simplify the Gamma function to yield the stated result.
\end{proof}
\begin{example}
The degenerate case.
\begin{equation}
\sum_{p=0}^{n-1}\tan \left(m+\frac{\pi  p q}{n}\right)=n \tan (m n)
\end{equation}
\end{example}
\begin{proof}
Use equation (\ref{fslf}) and set $k=0$ and simplify using entry (4) in table below (64:12:7) in \cite{atlas}.
\end{proof}
\begin{example}
\begin{multline}
\prod_{p=0}^{n-1}\cos ^3\left(\frac{\pi  p q}{n}+\frac{x}{2}\right) \sec ^2\left(\frac{\pi  p q}{n}+\frac{x}{4}\right) \sec
   \left(\frac{\pi  p q}{n}+x\right)\\\\
   =\cos ^3\left(\frac{n x}{2}\right) \sec ^2\left(\frac{n x}{4}\right) \sec (n
   x)
\end{multline}
\end{example}
\begin{proof}
Use equation (\ref{fslf}) and set $k=1,a=-1,m=x$ and apply the method in section (8) in \cite{reyn_ejpam}.
\end{proof}
\begin{example}
\begin{multline}
\prod_{p=0}^{n-1}e^{4 \tan \left(\frac{\pi  p q}{n}+x\right)-4 \tan \left(\frac{\pi  p q}{n}+\frac{x}{2}\right)} \left(\cos
   \left(\frac{\pi  p q}{n}+\frac{x}{2}\right) \sec \left(\frac{\pi  p q}{n}+x\right)\right)^{2 i \pi }\\\\
   =\left(\cos
   \left(\frac{n x}{2}\right) \sec (n x)\right)^{2 i \pi } e^{4 n \tan \left(\frac{n x}{2}\right) \sec (n x)}
\end{multline}
\end{example}
\begin{proof}
Use equation (\ref{fslf}) and set $k=1,a=i,m=x$ and apply the method in section (8) in \cite{reyn_ejpam}.
\end{proof}
\begin{example}
\begin{equation}
\prod_{p=0}^{n-1}\cos \left(m+\frac{\pi  p q}{n}\right) \sec \left(\frac{\pi  p q}{n}+r\right)=e^{i (n-1) (m-r)} \cos (m n) \sec
   (n r)
\end{equation}
\end{example}
\begin{proof}
Use equation (\ref{fslf}) and form a second equation by replacing $m$ by $r$ and take their difference and set $k=-1,a=1$ and simplify using entry (3) in Table below (64:12:7) in \cite{atlas}.
\end{proof}
\begin{example}
Catalan's constant $K$.
\begin{equation}
\sum_{p=0}^{n-1} n e^{\frac{2 i (\pi  p q+\pi )}{n}} \Phi \left(-e^{2 i \left(\frac{p \pi  q}{n}+\frac{\pi
   }{n}\right)},2,1-\frac{n}{2}\right)=4K
\end{equation}
\end{example}
\begin{proof}
Use equation (\ref{fslf}) and set $a=-n,m=\pi/n$ and simplify using equation(25.14.3) in \cite{dlmf} and equation (2.2.1.2.7) in \cite{lewin}.
\end{proof}
\begin{example}
Recurrence identity with consecutive neighbours 
\begin{multline}
\Phi (z,s,a)=-e^{\frac{2 i \pi  q}{3}} \Phi \left(e^{\frac{2 i \pi  q}{3}} z,s,a\right)-e^{\frac{4 i \pi  q}{3}}
   \Phi \left(e^{\frac{4 i \pi  q}{3}} z,s,a\right)+3^{1-s} z^2 \Phi \left(z^3,s,\frac{a+2}{3}\right)
\end{multline}
\end{example}
\begin{proof}
Use equation (\ref{fslf}) and set $n=3,m=\log(z)/(2i),k=-s,a=e^{2i(a-1)}$ and simplify. Where $z,s,a\in\mathbb{C},q\neq3n, n\in\mathbb{Z}$
\end{proof}
\section{Conclusion}
In this paper, we have presented a novel method for deriving a new finite sum of the Hurwitz-Lerch zeta function along with some interesting products of trigonometric functions and a new finite sum for Catalan's constant $K$, using contour integration. The results presented were numerically verified for both real and imaginary and complex values of the parameters in the integrals using Mathematica by Wolfram.

\begin{thebibliography}{999}
%
\bibitem{reyn4} Reynolds, R.; Stauffer, A.
{A Method for Evaluating Definite Integrals in Terms of Special Functions with Examples}.  \emph{Int. Math. Forum} \textbf{2020}, \emph{15}, 235--244, doi:10.12988/imf.2020.91272 
%
\bibitem{grad} Gradshteyn, I.S.; Ryzhik, I.M.
\emph{Tables of Integrals, Series and Products}, 6th ed.; Academic Press: Cambridge, MA, USA, 
 2000.
%
\bibitem{erd} Erd\'eyli, A.; Magnus, W.; Oberhettinger, F.; Tricomi, F.G.
\emph{Higher Transcendental Functions}; McGraw-Hill Book Company, Inc.: New York, NY, USA; Toronto, ON, Canada; London, UK, 1953; Volume I.
%
\bibitem{atlas} Oldham, K.B.; Myland, J.C.; Spanier, J.
\emph{An Atlas of Functions: With Equator, the Atlas Function Calculator}, 2nd ed.; Springer: New York, NY, USA, 2009.
%
\bibitem{dlmf} Olver, F.W.J.; Lozier, D.W.; Boisvert, R.F.; Clark, C.W. (Eds.)
 \emph{NIST Digital Library of Mathematical Functions}; U.S. Department of Commerce, National Institute of Standards and Technology: Washington, DC, USA; Cambridge University Press: Cambridge, UK, 2010; With 1 CD-ROM (Windows, Macintosh and UNIX). MR 2723248 (2012a:33001).
%
%
\bibitem{lewin}L. Lewin (1981) 
\emph{Polylogarithms and Associated Functions}. North-Holland Publishing Co., New York.
%
\bibitem{prud1}A. P. Prudnikov, Yu. A. Brychkov, and O. I. Marichev (1986a) 
\emph{Integrals and Series: Elementary Functions}, Vol. 1. Gordon \& Breach Science Publishers, New York. 
%
\bibitem{garunktis}
\emph{From Arithmetic to Zeta-Functions: Number Theory in Memory of Wolfgang Schwarz.} Germany: Springer International Publishing, (n.d.).
%
\bibitem{reyn_ejpam}Reynolds, R., \& Stauffer, A. (2022). 
\emph{A Note on the Infinite Sum of the Lerch function}. European Journal of Pure and Applied Mathematics, 15(1), 158–168. https://doi.org/10.29020/nybg.ejpam.v15i1.4137
%
\bibitem{gelca} Gelca, Răzvan., Andreescu, Titu. 
\emph{Putnam and Beyond}. Germany: Springer International Publishing, 2017.
 %
\end{thebibliography}
\end{document}